\documentclass[letterpaper, 10 pt, conference]{ieeeconf}  % Comment this line out
\IEEEoverridecommandlockouts                              % This command is only
\usepackage{amsmath} % assumes amsmath package installed
\usepackage{amssymb}  % assumes amsmath package installed

\usepackage{optidef}

\def\b{\ensuremath\boldsymbol}

\usepackage{url, cite}
\usepackage{balance}

\usepackage{color}

\usepackage[table]{xcolor}  % for color in table

\allowdisplaybreaks[4]

\newtheorem{theorem}{Theorem}
\newtheorem{lemma}{Lemma}
\newtheorem{definition}{Definition}

\newtheorem{assumption}{Assumption}
\newtheorem{proposition}{Proposition}
\newtheorem{corollary}{Corollary}
\newtheorem{remark}{Remark}

\title{\LARGE \bf
Distributed Nonlinear Model Predictive Control and Metric Learning for Heterogeneous Vehicle Platooning with Cut-in/Cut-out Maneuvers 
}

\author{Mohammad Hossein Basiri, Benyamin Ghojogh, Nasser L. Azad,~\IEEEmembership{Member,~IEEE}, \\
Sebastian Fischmeister,~\IEEEmembership{Member,~IEEE},
Fakhri Karray,~\IEEEmembership{Fellow,~IEEE}, Mark Crowley,~\IEEEmembership{Member,~IEEE} \\
\thanks{N. L. Azad is with the Department of Systems Design Engineering, University of Waterloo, ON, Canada. The other authors are with the Department of Electrical and Computer Engineering, University of Waterloo, ON, Canada. 
M. H. Basiri and N. L. Azad are with the Smart Hybrid and Electric Vehicles Systems lab, M. H. Basiri and S. Fischmeister are with the Real-time Embedded Software lab, B. Ghojogh and M. Crowley are with the Machine Learning lab, and F. Karray is with the centre for Pattern Analysis and Machine Intelligence.}%
\thanks{Email: \{mh.basiri, bghojogh, nlashgar, sfischme, karray, mcrowley\}@uwaterloo.ca}%
}

\usepackage{eso-pic} %--> used for the note above the first page

\begin{document}

% the note above the first page:
\AddToShipoutPictureBG*{%
  \AtPageUpperLeft{%
    \setlength\unitlength{1in}%
    \hspace*{\dimexpr0.5\paperwidth\relax}
    \makebox(0,-0.75)[c]{\normalsize Accepted for presentation at the 59th IEEE Conference on Decision and Control (CDC) 2020}
    }}

\maketitle
\thispagestyle{empty}
\pagestyle{empty}

%%%%%%%%%%%%%%%%%%%%%%%%%%%%%%%%%%%%%%%%%%%%%%%%%%%%%%%%%%%%%%%%%%%%%%%%%%%%%%%%
\begin{abstract}
Vehicle platooning has been shown to be quite fruitful in the transportation industry to enhance fuel economy, road throughput, and driving comfort. Model Predictive Control (MPC) is widely used in literature for platoon control to achieve certain objectives, such as safely reducing the distance among consecutive vehicles while following the leader vehicle. In this paper, we propose a Distributed Nonlinear MPC (DNMPC), based upon an existing approach, to control a heterogeneous dynamic platoon with unidirectional topologies, handling possible cut-in/cut-out maneuvers. The introduced method addresses a collision-free driving experience while tracking the desired speed profile and maintaining a safe desired gap among the vehicles. The time of convergence in the dynamic platooning is derived based on the time of cut-in and/or cut-out maneuvers. In addition, we analyze the improvement level of driving comfort, fuel economy, and absolute and relative convergence of the method by using distributed metric learning and distributed optimization with Alternating Direction Method of Multipliers (ADMM). Simulation results on a dynamic platoon with cut-in and cut-out maneuvers and with different unidirectional topologies show the effectiveness of the introduced method.
\end{abstract}

\section{Introduction}
Autonomous driving has been getting much attention during recent years due to its capability to improve the safe and reliable driving experience without the need for human resources. One of the significant advantages of these systems is to form a string of autonomous connected vehicles, called a platoon, all tracking the same shared speed profile generated by the Leader Vehicle (LV). The vehicles participating in a platoon can exchange data through various communication topologies such as Predecessor-Follower (PF), Predecessor-Leader Follower (PLF), Two Predecessors-Follower (TPF), and Two Predecessors-Leader Follower (TPLF) \cite{li2015overview} (see Fig. \ref{figure_topologies}).

\begin{figure}[!t]
\centering
\includegraphics[width=3in]{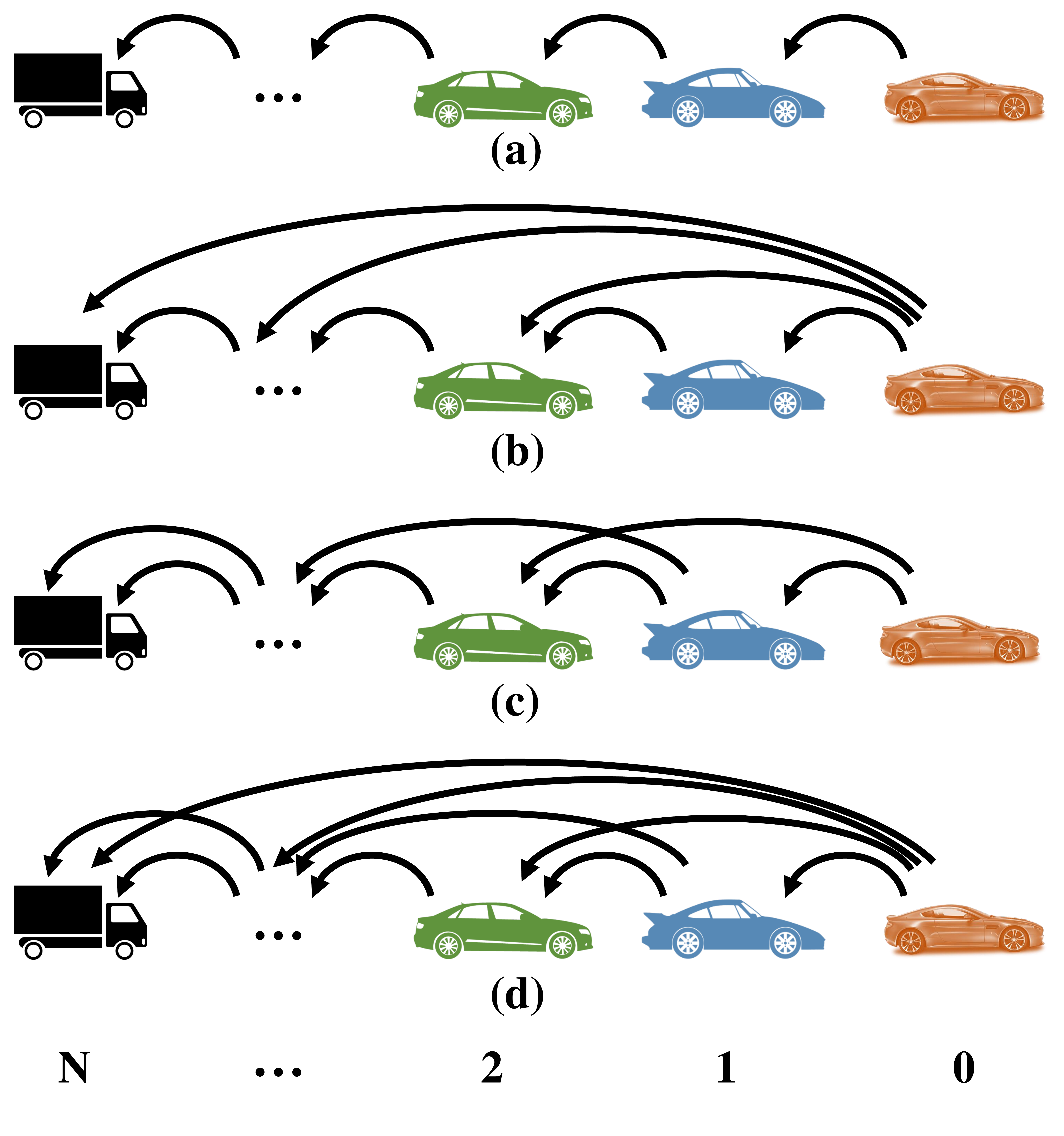}
\vspace{-.5cm}
\caption{Unidirectional topology: (a) PF, (b) PLF, (c) TPF, and (d) TPLF, (vehicle 0 is LV).}
\label{figure_topologies}
\end{figure}

In platoon control, it is required to guarantee specific control objectives such as the same velocity for the followers and a safe desired gap between any two consecutive vehicles. Over the years, researchers have introduced many different control techniques through Advanced Driver Assistance Systems (ADAS) to achieve desired control specifications \cite{li2019distributed,hu2020cooperative}. Commercially used Adaptive Cruise Control (ACC) and its more developed version, Cooperative ACC (CACC) \cite{milanes2016handling}, are to enhance traffic flow capacity. In this context, Model Predictive Control (MPC) is one of the prominent candidates to meet the control objectives while ensuring a reasonable amount of computational complexity \cite{dunbar2011distributed, sakhdari2018adaptive}. Basically, MPC formulates the control problem as an optimization problem with the control requirements as its constraints. 
Dynamic platooning is more challenging than static platooning because of possible cut-in/cut-out maneuvers \cite{kazerooni2015interaction,min2019constrained}. 
Most of the works in literature employ a linear system together with a linear control structure because of the complicated essence of dynamic platooning. Linear MPC \cite{shi2017distributed,kazemi2018learning}, PD \cite{lam2013cooperative,milanes2014cooperative}, and PID \cite{dasgupta2017merging} controllers are some examples.
Some works like \cite{goli2019mpc} have used linear approximations of the nonlinear model. 
Using a nonlinear model for dynamic control increases the complexity; however, exploiting the distributed version of NMPC can compensate for this complexity. 
Moreover, most of the works in dynamic platooning are proposed for homogeneous vehicles \cite{shi2017distributed,kazemi2018learning,lam2013cooperative,milanes2014cooperative,dasgupta2017merging,goli2019mpc}.
Also, some of the research works are solely able to handle cut-in (and not cut-out) maneuvers \cite{kazerooni2015interaction,dasgupta2017merging}.

Our contributions are explicitly two-fold.
Firstly, we propose a method, based on an existing Distributed Nonlinear MPC (DNMPC) \cite{zheng2017distributed}, for dynamic platooning. Our method is general for heterogeneous vehicles and can handle both cut-in and cut-out maneuvers while tracking the desired speed trajectory and ensuring the safe desired gap between any two consecutive vehicles. 
Secondly, we analyze driving experience, including driving comfort, fuel economy, and absolute and relative convergence by analyzing the metrics in DNMPC \cite{gechter2020platoon} using distributed metric learning \cite{kulis2013metric} and Alternating Direction Method of Multipliers (ADMM) optimization \cite{boyd2011distributed} which is found to be useful for MPC \cite{mota2012distributed}. 

The paper is organized as follows. In Section \ref{section_platoon_control}, we review the system modeling and the DNMPC proposed by \cite{zheng2017distributed}. 
Section \ref{section_preliminaries} introduces the preliminaries and background for stability, metric learning, and optimization. 
The extension of the DNMPC to the dynamic platooning is presented in Section \ref{section_dynamic_platoon}. 
Section \ref{section_metric_learning} analyzes the driving experience using distributed metric learning. Simulation results are provided in Section \ref{section_simulations}. Finally, Section \ref{section_conclusion} concludes the paper and provides some future directions.

\section{System Modeling}\label{section_platoon_control}

\subsection{Platoon Modeling}
Consider a platoon of vehicles, including an LV and $N$ Follower Vehicles (FVs) indexed by $\mathcal{N} := \{1, \dots, N\}$. In this paper, as in \cite{zheng2017distributed}, we consider the longitudinal dynamics and unidirectional communication topologies. 
Let $\Delta t$ be the discrete time interval and $s_i(t)$, $v_i(t)$, and $T_i(t)$ denote the position, velocity, and the integrated driving/breaking torque of the $i$-th FV at time $t$, respectively. 
For the $i$-th FV, we denote the vehicle's mass, the coefficient of aerodynamic drag, the coefficient of rolling resistance, the inertial lag of longitudinal dynamics, the tire radius, the mechanical efficiency of the driveline, and the control input by $m_i$, $C_{A,i}$, $f_i$, $\tau_i$, $r_i$, $\eta_i$, and $u_i(t) \in \mathbb{R}$, respectively. If $g$ is the gravity constant, the dynamics of the $i$-th FV are $\b{x}_i(t+1) = \b{\phi}_i(\b{x}_i(t)) + u_i(t)\, \b{\psi}_i$ and $\b{y}_i(t) = \b{\gamma}\, x_i(t)$ where $\b{x}_i(t) := [s_i(t), v_i(t), T_i(t)]^\top \in \mathbb{R}^3$ and $\b{y}(t) := [s_i(t), v_i(t)]^\top \in \mathbb{R}^2$ are the state and the control output, respectively, and $\b{\psi}_i := [0, 0, (1/\tau_i)\, \Delta t]^\top$, $\b{\gamma} := 
\begin{bmatrix}
1 & 0 & 0\\
0 & 1 & 0
\end{bmatrix}
$, and $\b{\phi}_i(\b{x}_i(t)) := [s_i(t) + v_i(t)\, \Delta t, v_i(t) + \frac{\Delta t}{m_i} (\frac{\eta_i}{r_i} T_i(t) - \Gamma_i(v_i(t))), T_i(t) - (1 / \tau_i)\, T_i(t)\, \Delta t]^\top$ where $\Gamma_i\big(v_i(t)\big) := C_{A,i}\, v_i^2(t) + m_i\, g\, f_i$ \cite{zheng2017distributed}. 

Let $\b{A} = [a_{ij}] \in \mathbb{R}^{N \times N}$ be the adjacency matrix where $a_{ij} = 1$ ($=0$) means that the $j$-th FV can (cannot) send information to the $i$-th FV. 
Also, let $p_i = 1$ ($=0$) mean that the $i$-th FV is (not) pinned to the LV and gets (does not get) information from it. Suppose $\mathbb{P}_i := \{0\}$ if $p_i = 1$ and $\mathbb{P}_i := \varnothing$ if $p_i = 0$.
We denote $\mathbb{N}_i := \{j | a_{ij} = 1, j \in \mathcal{N}\}$ and $\mathbb{O}_i := \{j | a_{ji} = 1, j \in \mathcal{N}\}$ as the sets of FVs which the $i$-th FV can get information from and send information to, respectively. The set $\mathbb{I}_i := \mathbb{N}_i \cup \mathbb{P}_i$ is the set of all vehicles sending information to the $i$-th FV. 

\subsection{Distributed Nonlinear Model Predictive Control}

\subsubsection{Preliminaries}

Let $s_0(t)$ and $v_0(t)$ be the position and velocity of LV, which is assumed to have a constant speed. 
For the $i$-th FV, the desired state and control signal are $\b{x}_{\text{des},i}(t) := [s_{\text{des},i}(t), v_{\text{des},i}(t), T_{\text{des},i}(t)]^\top$ and $u_{\text{des},i}(t) := T_{\text{des},i}(t)$, respectively, where $s_{\text{des},i}(t) := s_0(t) - i\, d_0$, $v_{\text{des},i}(t) := v_0$, $T_{\text{des},i}(t) := h_i(v_0)$ where $h_i(v_0) := (r_i / \eta_i) (C_{A,i}\, v_0^2 + m_i\, g\, f_i)$ is the external drag. The desired output is $\b{y}_{\text{des},i}(t) := \b{\gamma}\, \b{x}_{\text{des},i}(t) \in \mathbb{R}^2$. 
Note that Assumption \ref{assumption_spanning_tree}, introduced later in Section \ref{section_preliminaries}, is assumed here.

The goal of the control is to track the pace of the leader while having the desired distance between the vehicles in the platoon. In other words, we would like to have $\lim_{t \rightarrow \infty} |v_i(t) - v_0(t)| = 0$ and $\lim_{t \rightarrow \infty} |s_{i-1}(t) - s_i(t) - d| = 0$ where $d$ is the desired distance between every two consecutive vehicles \cite{zheng2017distributed}. We also denote the distance of the $i$-th and $j$-th FVs by $d_{i.j}$.

In this problem, we have two types of output, which are the predicted and the assumed outputs. 
The former is obtained by the calculated control input from optimization, which is fed to the system. The latter is obtained by shifting the optimal output of the last-step optimization problem. 
Let $\b{y}_i^p(k|t)$ and $\b{y}_i^a(k|t)$ denote the predicted output and the assumed output, respectively. 
We explain the calculation of these two outputs in the following sections. The predicted and assumed states are denoted by $\b{x}_i^p(k|t)$ and $\b{x}_i^a(k|t)$, respectively. 

\subsubsection{Optimization for Control}

Consider a predictive horizon $N_p$ for the predictive control of the platoon. 
Suppose the predicted control inputs over the horizon are $\mathcal{U}_i^p(t) := \{u_i^p(0|t), \dots, u_i^p(N_p-1|t)\}$ which need to be found by the follwoing optimization problem \cite{zheng2017distributed}: % https://tex.stackexchange.com/questions/387968/auto-numbering-constraints-in-optimization-problems
\begin{mini!}|l|[2]             
    {\mathcal{U}_i^p(t)}                 
    {J_i(\b{y}_i^p, u_i^p, \b{y}_i^p, \b{y}_{-i}^p) \label{equation_optimization_DNMPC_obj}}  
    {\label{equation_optimization_DNMPC}}             
    {}                                
    \addConstraint{\b{x}_i^p(k+1|t)\!}{=\! \b{\phi}_i(\b{x}_i^p(k|t))\! + u_i^p(k|t)\, \b{\psi}_i \label{equation_optimization_DNMPC_cons1}}    
    \addConstraint{\b{y}_i^p(k|t)}{= \b{\gamma}\, \b{x}_i^p(k|t) \label{equation_optimization_DNMPC_cons2}}  
    \addConstraint{\b{x}_i^p(0|t)}{= \b{x}_i(t) \label{equation_optimization_DNMPC_cons3}} 
    \addConstraint{u_i^p(k|t)}{\in [u_{\min,i}, u_{\max,i}]
    \label{equation_optimization_DNMPC_cons4}} 
    \addConstraint{\b{y}_i^p(N_p|t)}{= \frac{1}{|\mathbb{I}_i|} \sum_{j \in \mathbb{I}_i} \big(\b{y}_j^a(N_p|t) + \b{\widetilde{d}}_{i,j}\big)
    \label{equation_optimization_DNMPC_cons5}} 
    \addConstraint{T_i^p(N_p | t)}{= h_i(v_i^p(N_p|t))
    \label{equation_optimization_DNMPC_cons6}},
\end{mini!}
where $\b{y}_{-i}(t) := [\b{y}_{i_1}^\top, \dots, \b{y}_{i_m}^\top]^\top$ (if $\{i_1, \dots, i_m\} := \mathbb{N}_i$), $u_{\min,i}$ and $u_{\max,i}$ are the lower and upper bounds of the control input, $|\mathbb{I}_i|$ is the cardinality of $\mathbb{I}_i$, and $\b{\widetilde{d}}_{i,j} := [d_{i.j}, 0]^\top$. 
For the intuitions of the constraints in Eq. (\ref{equation_optimization_DNMPC}), refer to \cite{zheng2017distributed}. 

The objective function (\ref{equation_optimization_DNMPC_obj}) is the summation of local cost functions as a Lyapunov function \cite{zheng2017distributed}: 
\begin{align}\label{equation_objective}
&J_i(\b{y}_i^p, u_i^p, \b{y}_i^a, \b{y}_{-i}^a) := \sum_{k=0}^{N_p - 1} \Big( \|\b{y}_i^p(k|t) - \b{y}_{\text{des},i}(k|t)\|_{\b{Q}_i} \nonumber \\
&~~~~~ + \|u_i^p(k|t) - h_i(v_i^p)\|_{R_i} + \|\b{y}_i^p (k|t) - \b{y}_i^a(k|t)\|_{\b{F}_i} \nonumber \\
&~~~~~ + \sum_{j \in \mathbb{N}_i} \|\b{y}_i^p (k|t) - \b{y}_j^a(k|t) + \b{\widetilde{d}}_{i,j}\|_{\b{G}_i} \Big),
\end{align}
in which, for a weight matrix $\b{A} \succeq 0$ (where $\succeq 0$ means belonging to the positive semi-definite cone), we define $\|\b{x}\|_{\b{A}} := \b{x}^\top \b{A}\, \b{x}$.
In Eq. (\ref{equation_objective}), we have $0 \preceq \b{Q}_i, \b{F}_i, \b{G}_i \in \mathbb{R}^{2 \times 2}$ and $0 \leq R_i \in \mathbb{R}$ are the weight matrices which act as regularization. The matrices $\b{Q}_i$, $R_i$, $\b{F}_i$, $\b{G}_i$ regularize the amount of penalty for deviation from the desired output $\b{y}_{\text{des},i}(k|t)$, deviation of the control input from the equilibrium, deviation from the assumed output, and deviation from the trajectories of vehicle's neighbors, respectively.
The optimization problem (\ref{equation_optimization_DNMPC}) can be solved using the interior point method \cite{boyd2004convex}. 
For the algorithm of the DNMPC using Optimization (\ref{equation_optimization_DNMPC}), please refer to \cite{zheng2017distributed}.

\section{Preliminaries and Background}\label{section_preliminaries}

\subsection{Preliminaries for Stability}

In the following, the preliminary lemmas, from the baseline paper \cite{zheng2017distributed}, are mentioned. These lemmas are required for the next coming theories on stability of the proposed dynamic platooning. 

\begin{assumption}\label{assumption_spanning_tree}
The directed graph of the platoon topology contains a spanning tree rooted at the LV. 
This assumption is necessary for stability in both homogeneous \cite{zheng2016stability} and heterogeneous \cite{zheng2017distributed} platooning. This ensures that all vehicles get the leader's information either directly or indirectly. 
\end{assumption}

\begin{lemma}[\!\!{\cite[Theorem 2]{zheng2017distributed}}]\label{lemma_DNMPC_convergence_in_N_steps}
If Assumption \ref{assumption_spanning_tree} is satisfied, then Problem (\ref{equation_optimization_DNMPC}) guarantees convergence of the output to the desired output in at most $N$ time steps, i.e., $\b{y}_i^p(N_p | t) = \b{y}_{\text{des},i}(N_p | t), \forall t \geq N$, for a static platoon (without any dynamic maneuvers).
\end{lemma}

\begin{lemma}[\!\!{\cite[Theorems 3 and 4]{zheng2017distributed}}]\label{lemma_mainProblem_Lyapunov}
The objective function in Problem (\ref{equation_optimization_DNMPC}) is the sum of non-increasing Lyapunov functions; hence, this system is asymptotically stable. 
\end{lemma}

\begin{lemma}[\!\!{\cite[Theorem 5]{zheng2017distributed}}]\label{lemma_stability_sufficient_condition}
Satisfying Assumption \ref{assumption_spanning_tree}, a sufficient condition for asymptotic stability of the platoon is $\b{F}_i - \sum_{j \in \mathbb{O}_i} \b{G}_j \succeq \b{0}, \forall i \in \mathcal{N}$. 
\end{lemma}

\subsection{Preliminaries for Metrics and Metric Learning}

Metric learning is a tool of machine learning for obtaining a promising metric subspace for better representation of data \cite{kulis2013metric}.
In the following, we provide some definitions and lemmas regarding metrics and metric learning. 

\begin{definition}[metric distance]
A metric distance between $\b{x}_1$ and $\b{x}_2$ can be written as \cite{kulis2013metric}:
\begin{align}\label{equation_metric_distance}
\|\b{x}_1 - \b{x}_2\|_{\b{A}} &= (\b{x}_1 - \b{x}_2)^\top \b{A}\, (\b{x}_1 - \b{x}_2),
\end{align}
where $\b{A} \succeq 0$ is the weight matrix.
\end{definition}

\begin{lemma}[\!\!\cite{kulis2013metric,boyd2004convex}]\label{lemma_metric_matrix_positive_definite}
A necessary and sufficient condition for Eq. (\ref{equation_metric_distance}) to be a valid convex distance metric, satisfying the triangle inequality, is to have positive semi-definite weight matrix, i.e., $\b{A} \succeq 0$.
\end{lemma}

% Metric learning is the task of learning the weight matrix $\b{A}$ to have a suitable metric for the goal \cite{kulis2013metric}, which is the platoon control here.

\begin{lemma}\label{lemma_metric_subspace_projection}
A metric can be seen as projection onto a (lower dimensional) subspace of its factorized weight matrix.
\end{lemma}
\begin{proof}
The metric can be stated as:
\begin{align}
\|\b{x}_1 - &\b{x}_2\|_{\b{A}} \overset{(\ref{equation_metric_distance})}{=} (\b{x}_1 - \b{x}_2)^\top \b{A}\, (\b{x}_1 - \b{x}_2) \nonumber \\
&\overset{(a)}{=} (\b{B}^\top \b{x}_1 -\! \b{B}^\top \b{x}_2)^\top (\b{B}^\top \b{x}_1 -\! \b{B}^\top \b{x}_2), \label{equation_metric_projection}
\end{align}
where $(a)$ is because $\b{A} \succeq 0$ and $\b{A} = \b{A}^\top$ so its factorization (eigenvalue decomposition) can be written as: $\b{A} = \b{\Psi}\, \b{\Xi}\, \b{\Psi}^\top = \b{\Psi}\, \b{\Xi}^{(1/2)}\, \b{\Xi}^{(1/2)}\, \b{\Psi}^\top = \b{B}\b{B}^\top$, by taking $\b{B} := \b{\Psi}\, \b{\Xi}^{(1/2)}$. 
The Eq. (\ref{equation_metric_projection}) is the Euclidean distance metric in the column space of $\b{B}$. 
\end{proof}

\begin{definition}[$\varepsilon$-positive definite]\label{definition_epsilon_positive_definite}
A symmetric matrix $\b{A} \in \mathbb{R}^{n \times n}$ is $\varepsilon$-positive definite if and only if:
\begin{align}\label{equation_epsilon_positive_definite}
\b{x}^\top \b{A}\, \b{x} > \varepsilon, \quad \forall \b{x} \in \mathbb{R}^{n} \setminus \{\b{0}\},
\end{align}
for $\varepsilon\!\! >\!\! 0$.
We denote the $\varepsilon$-positive definite cone and belonging to this cone by $\mathbb{S}_\varepsilon^n$ and $\b{A} \succ_{\varepsilon} \b{0}$, respectively. 
\end{definition}

% \begin{corollary}
% For $\b{A} \succ_{\varepsilon} \b{0}$, we have:
% \begin{align}\label{equation_epsilon_positive_definite_delta}
% \b{A} \succ \delta \b{I}, \text{ where } \varepsilon =: \b{x}^\top \delta\, \b{x}, \forall \b{x} \in \mathbb{R}^{n} \setminus \{\b{0}\},
% \end{align}
% where $\b{I}$ is the identity matrix. 
% \end{corollary}
% \begin{proof}
% According to Definition \ref{definition_epsilon_positive_definite}, we have:
% \begin{align*}
% &\b{A} \succ_{\varepsilon} \b{0} \implies \b{x}^\top \b{A}\, \b{x} > \varepsilon \overset{(\ref{equation_epsilon_positive_definite_delta})}{\implies} \b{x}^\top \b{A}\, \b{x} > \b{x}^\top \delta\, \b{x} \\
% &\therefore~~ \b{x}^\top (\b{A} - \delta \b{I})\, \b{x} > \b{0} \overset{(a)}{\implies} \b{A} - \delta \b{I} \succ \b{0},
% \end{align*}
% where $(a)$ is according to the definition of positive definite matrix. 
% \end{proof}

\subsection{Preliminaries for Optimization}

In the following, we provide some preliminary theories for optimization and projection onto the convex sets. 

\begin{lemma}[\!\!\cite{parikh2014proximal}]
Consider the proximal operator, defined as:
\begin{align}
\textbf{prox}_f(\b{A}) := \arg\min_{\b{B}} \big(f(\b{B}) - \frac{1}{2} \|\b{B} - \b{A}\|_F^2\big),
\end{align}
where $\|.\|_F$ denotes the Frobenius norm. If the function $f$ is an indicator function, $\mathcal{I}(\mathcal{S})$, which is equal to zero and infinity when its input belongs and does not belong to the set $\mathcal{S}$, respectively, the operator is reduced to projection on the set, denoted by:
\begin{align}\label{equation_projection}
\Pi_{\mathcal{S}}(\b{A}) := \arg\min_{\b{B}\in \mathcal{S}}\|\b{B} - \b{A}\|_F^2.
\end{align}
\end{lemma}

% Projection onto cones is useful in optimization \cite{henrion2012projection}. 
The following theorem explains how to project a symmetric matrix onto the $\varepsilon$-positive definite cone.
\begin{theorem}\label{theorem_projection_epsilon_positive_definite}
Projection of a matrix $\b{A} \in \mathbb{R}^{n \times n}$ onto the $\varepsilon$-positive definite cone is:
\begin{align}
\Pi_{\mathbb{S}_\varepsilon^n}(\b{A}) := \b{V}\, \textbf{diag}(\max(\b{\Upsilon}, \varepsilon))\, \b{V}^\top,
\end{align}
where $\b{A} = \b{V} \b{\Upsilon} \b{V}^\top$ is the eigenvalue decomposition of $\b{A}$ and $\textbf{diag}(.)$ is the diagonal matrix with its argument as the diagonal. 
\end{theorem}
\begin{proof}
According to Eq. (\ref{equation_projection}), we have: $\Pi_{\mathbb{S}_\varepsilon^n}(\b{A}) = \arg\min_{\b{X} \in \mathbb{S}_\varepsilon^n}\|\b{X}-\b{A}\|^2_F$. 
We have $\|\b{X}-\b{A}\|^2_F = \|\b{X} - \b{V} \b{\Upsilon} \b{V}^\top\|^2_F \overset{(a)}{=} \|\b{V} (\b{V}^\top \b{X} \b{V} - \b{\Upsilon}) \b{V}^\top\|^2_F \overset{(b)}{=} \textbf{tr}\big((\b{V}^\top \b{V})^2 (\b{V}^\top \b{X} \b{V} - \b{\Upsilon})^2 \big) \overset{(c)}{=} \textbf{tr}(\b{\Upsilon}^2) + \textbf{tr}((\b{V}^\top \b{X} \b{V})^2) -2\, \textbf{tr}(\b{\Upsilon} \b{V}^\top \b{X} \b{V})$ where $\textbf{tr}(.)$ denotes the trace of matrix, $(a)$ and $(c)$ are because $\b{V}$ is an orthogonal matrix and $(b)$ is because the term in the norm is symmetric. 
% This approach also exists in multi-dimensional scaling in manifold learning \cite{cox2000multidimensional}. 
For the sake of minimization, we have: $\partial \|\b{X}-\b{A}\|^2_F / \partial \b{X} = 2\b{V} (\b{V}^\top \b{X} \b{V}) \b{V}^\top - 2 \b{V} \b{\Upsilon} \b{V}^\top \overset{\text{set}}{=} \b{0} \implies \b{V} (\b{V}^\top \b{X} \b{V}) \b{V}^\top = \b{V} \b{\Upsilon} \b{V}^\top \implies \b{\Upsilon} = \b{V}^\top \b{X} \b{V}$.
Considering the columns of $\b{V}$ (which are orthonormal so are not zero vectors) as $\b{x}$ in Eq. (\ref{equation_epsilon_positive_definite}) gives $\b{\Upsilon} = \b{V}^\top \b{X} \b{V} > \varepsilon \b{I}$.
As the eigenvalues of the diagonal matrix $\b{\Upsilon}$ are its diagonal entries, projection of $\b{A}$ onto the set $\mathbb{S}_\varepsilon^n$ clips the eigenvalues, i.e. diagonal entries of $\b{\Upsilon}$, to $\varepsilon$. 
\end{proof}

In the following, we provide corollaries about the convexity and characteristics of the $\varepsilon$-positive definite cone. 

\begin{corollary}\label{corollary_espilon_positive_definite_subset}
According to the Definition \ref{definition_epsilon_positive_definite} and Theorem \ref{theorem_projection_epsilon_positive_definite}, the eigenvalues of an $\varepsilon$-positive definite matrix are at least equal to $\varepsilon$. Hence, as $\varepsilon \neq 0$, the $\varepsilon$-positive definite cone is an inclusive subset of the positive semi-definite cone, i.e., $\mathbb{S}_\varepsilon^n \subset \mathbb{S}_+^n$. 
\end{corollary}

\begin{corollary}\label{corollary_metric_matrix_epsilon_positive_definite}
According to Lemma \ref{lemma_metric_matrix_positive_definite} and Corollary \ref{corollary_espilon_positive_definite_subset}, if the weight matrix belongs to the $\varepsilon$-positive definite cone, the metric is valid. 
\end{corollary}

\begin{lemma}[\!\!\cite{boyd2004convex}]\label{lemma_convex_objective_function}
The objective function in Problem (\ref{equation_optimization_DNMPC}), which is a non-negative combination of some quadratic convex functions (according to Eq. (\ref{equation_objective})), is a convex function. 
\end{lemma}

\section{Dynamic Platoon Control: Handling Cut-in/Cut-out Maneuvers}\label{section_dynamic_platoon}

In this section, we propose the extension of the DNMPC \cite{zheng2017distributed} for handling the dynamic cut-in/cut-out maneuvers. Assume there exist $N_{ci}$ cut-in and $N_{co}$ cut-out maneuvers in total while the number of initial FVs in the platoon is $N$. Let $\mathcal{N}_{ci} := \{1, \dots, N_{ci}\}$ and $\mathcal{N}_{co} := \{1, \dots, N_{co}\}$. We denote the time of the $i$-th cut-in and the $j$-th cut-out maneuvers by $t_{ci,i}$ and $t_{co,j}$, respectively. The following theorem determines the time of convergence of a dynamic platoon including possible cut-in and cut-out maneuvers. 

\begin{theorem}\label{theorem_DNMPC_convergence_in_N_steps_dynamic}
When having cut-in and cut-out maneuvers, if Assumption \ref{assumption_spanning_tree} is satisfied, the Problem (\ref{equation_optimization_DNMPC}) guarantees convergence of the output to the desired output in at most
\begin{equation}\label{equation_DNMPC_convergence_in_N_steps_dynamic_time}
\begin{aligned}
t_\text{conv} :=  &\,\underset{i,j}{\max}\big[t_{ci,i}, t_{co,j} \,|\, \forall i \in \mathcal{N}_{ci}, \forall j \in \mathcal{N}_{cj}\big] \\
&+ N + N_{ci} - N_{co},
\end{aligned}
\end{equation}
time steps, i.e., $\b{y}_i^p(N_p | t) = \b{y}_{\text{des},i}(N_p | t), \forall t \geq t_\text{conv}$.
\end{theorem}
\begin{proof}
Let $\b{P} := \textbf{diag}(p_1, \dots, p_N)$, and $\b{D}$, $\b{A}$, and $\b{L} := \b{D} - \b{A}$ be the degree matrix, adjacency matrix, and Laplacian matrix of the platoon graph, respectively. 
When a new cut-in or cut-out occurs, some new chaos is introduced to the system so we can consider the latest cut-in/cut-out maneuver. Considering the latest cut-in, one vehicle is added to the number of existing vehicles denoted by $N$. If the platoon graph is unidirectional and satisfies Assumption \ref{assumption_spanning_tree}, the new $\b{A} \in \mathbb{R}^{(N + 1) \times (N + 1)}$ is a lower-triangular matrix. Moreover, according to {\cite[Lemma 4]{zheng2017distributed}}, we have $\b{D} + \b{P} > 0$, yielding the eigenvalues of $(\b{D} + \b{P})^{-1} \b{A}$ to be zero and this matrix to be nilpotent with degree at most $N + 1$. Based on {\cite[Lemma 1]{zheng2017distributed}} and {\cite[Theorem 1]{zheng2017distributed}}, $\b{y}_i^p(N_p | t)$ converges to the desired output in at most $N + 1$ steps. 
Extending this to $N_{ci}$ cut-in maneuvers requires $N + N_{ci}$ time steps after the latest cut-in. Similar analysis can be performed for the cut-out maneuvers, resulting in $N - N_{co}$ time steps after the latest cut-out because the number of vehicles has been reduced. In general, having $N_{ci}$ cut-in and $N_{co}$ cut-out maneuvers will need $N + N_{ci} - N_{co}$ time steps after the latest maneuver which is formulated as $\max_{i,j}[t_{ci,i}, t_{co,j} \,|\, \forall i \in \mathcal{N}_{ci}, \forall j \in \mathcal{N}_{cj}]$.
\end{proof}

\begin{corollary}
Lemma \ref{lemma_DNMPC_convergence_in_N_steps}, for the static platoon, is a special case of Theorem \ref{theorem_DNMPC_convergence_in_N_steps_dynamic} which is for a dynamic platoon. 
\end{corollary}
\begin{proof}
When neither cut-in nor cut-out happens, the time of convergence is $t_\text{conv} = 0 + N + 0 + 0 = N$ according to Eq. (\ref{equation_DNMPC_convergence_in_N_steps_dynamic_time}). 
\end{proof}

\begin{remark}
Two extreme special cases of the dynamic platoon are as the following examples:

Example 1) One cut-in at $t=0$ and one cut-out at $t=N$: According to Eq. (\ref{equation_DNMPC_convergence_in_N_steps_dynamic_time}), the platoon converges in $t = N + N + 1 - 1 = 2 N$. It is correct because before the cut-out, the platoon contains $N + 1$ vehicles until time $N$. When cut-out happens, the platoon is changed to a platoon with $N$ vehicles which converges in $N$ time steps according to Lemma \ref{lemma_DNMPC_convergence_in_N_steps}. 

Example 2) One cut-in at $t=N$ and one cut-out at $t=0$: According to Eq. (\ref{equation_DNMPC_convergence_in_N_steps_dynamic_time}), the platoon converges in $t = N + N + 1 - 1 = 2 N$ which is correct because in $t \in [0, N]$, the platoon includes $N - 1$ vehicles until time $N$. When cut-in happens, the platoon is modified to a platoon with $N$ vehicles which converges in $N$ time steps according to Lemma \ref{lemma_DNMPC_convergence_in_N_steps}. 
\end{remark}

\section{Driving Experience Analysis Using Distributed Metric Learning }\label{section_metric_learning}

The paper \cite{zheng2017distributed} uses manual metrics in the DNMPC. 
The extension of the DNMPC for the dynamic maneuvers, reported in the previous section, uses manual metrics satisfying Lemmas \ref{lemma_stability_sufficient_condition} and \ref{lemma_metric_matrix_positive_definite}. 
However, for the sake of analyzing the behaviors of metrics in DNMPC, we can learn the subspaces of metrics on which the data can be projected (see Lemma \ref{lemma_metric_subspace_projection}). The subspaces can properly show the difference of the predicted and assumed variables in the DNMPC for driving experience analysis. 
For learning the metrics, ADMM can be used as a distributed optimization tool. In the following, we first propose the distributed optimization for the metric learning in DNMPC. Thereafter, we provide its solution using ADMM. 
Note that it is essential that ADMM should not change the convergence behavior of the DNMPC for dynamic maneuvers so that the behavior analysis is valid for the proposed methodology. 
Using the theory of optimization, we provide the theory supporting this claim (see Theorem \ref{theorem_ML_convergence_in_N_steps}). 

\subsection{Optimization}

To formulate the optimization problem for distributed metric learning in platooning, we first provide some required details.

\begin{remark}\label{remark_belong_to_epsilon_positive_definite_in_optimization}
If the weight matrices belong to the positive semi-definite cone, many of them will tend to become zero matrices gradually in optimizing Eq. (\ref{equation_objective}). This is because in Eq. (\ref{equation_metric_distance}), the metric becomes minimized, i.e. zero, by making the weight matrix zero. 
To avoid this, the weight matrices should belong to the $\varepsilon$-positive definite cone (which is fine because of Corollary \ref{corollary_metric_matrix_epsilon_positive_definite}), so the optimization problem concentrates on a valid gradient direction. 
\end{remark}

\begin{corollary}\label{corollary_constraint_F}
According to Corollary \ref{corollary_metric_matrix_epsilon_positive_definite}, Remark \ref{remark_belong_to_epsilon_positive_definite_in_optimization}, and Lemma \ref{lemma_stability_sufficient_condition}, the weight matrix $\b{F}_i$, for the $i$-th FV, should satisfy Eq. (\ref{equation_optimization_ML_cons8}) to have asymptotic stability of the platoon. 
\end{corollary}

According to Remark \ref{remark_belong_to_epsilon_positive_definite_in_optimization} and Corollary \ref{corollary_constraint_F}, the weight matrices in the metrics should satisfy $\b{Q}_i \succeq_\varepsilon \b{0}$, $R_i \geq \varepsilon$, $\b{G}_i \succeq_\varepsilon \b{0}$, $\b{F}_i \succeq_\varepsilon \b{0}$, and $\b{F}_i - \sum_{j \in \mathbb{O}_i} \b{G}_j \succeq \b{0}, \forall i \in \mathcal{N}$. 
Moreover, as seen in Eq. (\ref{equation_objective}), the term with $\b{Q}_i$ requires connection to LV (i.e., $p = 1$) to have the desired output and the term with $\b{G}_i$ needs $\mathbb{N}_i \neq \varnothing$. Hence, if not satisfying these conditions, they must be zero matrices. 
Combining all these constraints with Problem (\ref{equation_optimization_DNMPC}) results in the following optimization problem, where a new objective variable, $\b{\Theta}_i \in \mathbb{R}^2$, is added to make the problem distributed \cite{boyd2011distributed}:
\begin{mini!}|l|[2]             
    {\mathcal{U}_i^p(t), \b{Q}_i, \b{\Theta}_i, R_i, \b{F}_i, \b{G}_i}                 
    {J_i(y_i^p, u_i^p, y_i^a, y_{-i}^a) \label{equation_optimization_ML_obj}}  
    {\label{equation_optimization_ML}}             
    {}                                
    \addConstraint{\hspace{-1cm}\text{Constraints }}{ \text{(\ref{equation_optimization_DNMPC_cons1})} \text{ to } \text{(\ref{equation_optimization_DNMPC_cons6})} \label{equation_optimization_ML_cons1}}    
    \addConstraint{\hspace{-2.5cm}\b{\Theta}_i}{\succeq_\varepsilon \b{0} \label{equation_optimization_ML_cons2}}  
    \addConstraint{\hspace{-2.5cm}\b{Q}_i}{= \b{0}, \quad \text{if } p_i = 0 \label{equation_optimization_ML_cons3}} 
    \addConstraint{\hspace{-2.5cm}\b{Q}_i - \b{\Theta}_i}{= \b{0} \label{equation_optimization_ML_cons4}}  
    \addConstraint{\hspace{-2.5cm}R_i}{\geq \varepsilon
    \label{equation_optimization_ML_cons5}} 
    \addConstraint{\hspace{-2.5cm}}{
    \left\{
        \begin{array}{ll}
            \b{G}_i \succeq_\varepsilon \b{0} & \text{if } \mathbb{N}_i \neq \varnothing \\
            \b{G}_i = \b{0} & \text{O.W.}
        \end{array}
    \right.
    \label{equation_optimization_ML_cons7}}
    \addConstraint{\hspace{-2.5cm}}{
    \left\{
        \begin{array}{ll}
            \b{F}_i \succeq_\varepsilon \b{0} & \text{if } \mathbb{O}_i = \varnothing \\
            \b{F}_i - \sum_{j \in \mathbb{O}_i} \b{G}_j \succeq \b{0} & \text{O.W.}
        \end{array}
    \right.
    \label{equation_optimization_ML_cons8}},
\end{mini!}

% \addConstraint{\hspace{-2cm}\b{F}_i - \sum_{j \in \mathbb{O}_i} \b{G}_j}{\succeq \b{0}

\begin{proposition}\label{lemma_strong_duality}
The Problem (\ref{equation_optimization_ML}) has strong duality. 
\end{proposition}
\begin{proof}
According to Lemma \ref{lemma_convex_objective_function}, the objective function in Problem (\ref{equation_optimization_ML}), as in Problem (\ref{equation_optimization_DNMPC}), is convex. 
As the constraints belong to the positive semi-definite or $\varepsilon$-positive definite cones, the problem is convex. 
According to Slater's condition on convex problems \cite{slater2014lagrange}, strong duality is guaranteed in Problem (\ref{equation_optimization_ML}) {\cite[Proposition 5.2.1]{bertsekas1999nonlinear}}.
\end{proof}

\subsection{Solving Optimization with ADMM}

The updates of objective variables are as the following, according to ADMM \cite{boyd2011distributed}:
\begin{subequations}\label{equation_ADMM_update}
\begin{align}
& \mathcal{U}_i^{p, (\kappa+1)}(t) := \text{Solution to Problem \! (\ref{equation_optimization_DNMPC})} \nonumber \\ 
& \qquad\qquad\qquad\qquad\qquad \text{with } \b{Q}_i^{(\kappa)}\!, R_i^{(\kappa)}\!, \b{F}_i^{(\kappa)}\!, \b{G}_i^{(\kappa)}\!, \label{equation_ADMM_update_U} \\
&\b{Q}_i^{(\kappa+1)} := 
\left\{
    \begin{array}{ll}
        \b{0} & \text{if } p_i = 0, \\
        \arg \min_{\b{Q}_i} \big( J_i + (\rho/2)\, \|\b{Q}_i \\
        ~~~~~~~~~~~~ - \b{\Theta}_i^{(\kappa)} 
         + \b{\Omega}_i^{(\kappa)}\|_F^2 \big) & \text{O.W.}
    \end{array}
\right. \label{equation_ADMM_update_Q} \\
&\b{\Theta}_i^{(\kappa+1)} := \arg \min_{\b{\Theta}_i} \Big( (\rho/2)\, \|\b{Q}_i^{(\kappa+1)} - \b{\Theta}_i + \b{\Omega}_i^{(\kappa)}\|_F^2, \nonumber \\ 
&~~~~~~~~~~~~~~~~~~~~~~~~~ \text{s.t. } \b{\Theta}_i \succeq_\varepsilon \b{0} \Big), \label{equation_ADMM_update_Theta} \\
&\b{\Omega}_i^{(\kappa+1)} := \b{\Omega}_i^{(\kappa)} + \b{Q}_i^{(\kappa+1)} - \b{\Theta}_i^{(\kappa+1)}, \label{equation_ADMM_update_Omega} \\
&R_i^{(\kappa+1)} := \arg \min_{R_i} \big( J_i, \text{ s.t. } R_i \geq \varepsilon \big), \label{equation_ADMM_update_R} \\
&\b{G}_i^{(\kappa+1)} := \arg \min_{\b{G}_i} \big( J_i, \text{ s.t. Constraint (\ref{equation_optimization_ML_cons7})} \big), \label{equation_ADMM_update_G} \\
&\b{F}_i^{(\kappa+1)} := \arg \min_{\b{F}_i} \big( J_i, \text{ s.t. Constraint (\ref{equation_optimization_ML_cons8})}\big), \label{equation_ADMM_update_F}
\end{align}
\end{subequations}
where $\kappa$ denotes the iteration of ADMM. 
We use the projected gradient method \cite{boyd2004convex} to solve the Eqs. (\ref{equation_ADMM_update_Q}), (\ref{equation_ADMM_update_Theta}), and (\ref{equation_ADMM_update_R})--(\ref{equation_ADMM_update_F}). We also use $\|\b{x}\|_{\b{A}} = \b{x}^\top \b{A} \b{x} = \textbf{tr}(\b{x}^\top \b{A} \b{x}) = \textbf{tr}(\b{x} \b{x}^\top \b{A})$, $\|\b{A}\|_F^2 = \textbf{tr}(\b{A}^\top \b{A})$, and $\partial \|\b{x}\|_{\b{A}} / \partial \b{A} = \b{x} \b{x}^\top$ for calculation of gradients. For projection onto the sets of $\varepsilon$-positive definite cone in the projected gradient method, we use Theorem \ref{theorem_projection_epsilon_positive_definite}. To update the control inputs, by Eq. (\ref{equation_ADMM_update_U}), we use the interior point method \cite{boyd2004convex} as also done in \cite{zheng2017distributed}. 

\begin{theorem}\label{theorem_ADMM_convergence}
The solution of Problem (\ref{equation_optimization_ML}) by iteratively performing Eq. (\ref{equation_ADMM_update}) converges to the optimal value satisfying the Karush-Kuhn-Tucker (KKT) conditions. 
\end{theorem}
\begin{proof}
We put together all the equality constraints and also inequality constraints amongst Eqs. (\ref{equation_optimization_ML_cons1}) and (\ref{equation_optimization_ML_cons5})--(\ref{equation_optimization_ML_cons8}), which are constant w.r.t. $\b{Q}_i$, $\b{\Theta}_i$, and $\b{\Omega}_i$. Let $\b{\Sigma}_{i,1}$ and $\b{\Sigma}_{i,2}$ denote these two groups of constraints, respectively. 
The augmented Lagrangian of Problem (\ref{equation_optimization_ML}) is \cite{boyd2011distributed,giesen2018distributed}:
\begin{equation*}
\begin{aligned}
&\mathcal{L}_\rho = J_i(y_i^p, u_i^p, y_i^p, y_{-i}^p) + \textbf{tr}\big(\b{\Lambda}_i^\top (\b{Q}_i - \b{\Theta}_i)\big) \\
&+ (\rho/2)\, \|\b{Q}_i - \b{\Theta}_i\|_F^2 + \textbf{tr}(\b{\Delta}_{i,1}^\top \b{\Sigma}_{i,1}) + \textbf{tr}(\b{\Delta}_{i,2}^\top \b{\Sigma}_{i,2}) \\
&= J_i(y_i^p, u_i^p, y_i^p, y_{-i}^p) + (\rho/2)\, \|\b{Q}_i - \b{\Theta}_i + \b{\Omega}_i\|_F^2 \\
&- (1/(2\rho))\, \|\b{\Lambda}_i\|_F^2 + \textbf{tr}(\b{\Delta}_{i,1}^\top \b{\Sigma}_{i,1}) + \textbf{tr}(\b{\Delta}_{i,2}^\top \b{\Sigma}_{i,2}),
\end{aligned}
\end{equation*}
where $\b{\Lambda}_i, \b{\Sigma}_{i,1}, \b{\Sigma}_{i,2} \in \mathbb{R}^{2 \times 2}$ are the Lagrange multipliers, $\rho\! >\! 0$ is the parameter of $\mathcal{L}_\rho$, and $\b{\Omega}_i := (1/\rho) \b{\Lambda}_i$ is the dual variable. Note that the term $(1/(2\rho))\, \|\b{\Lambda}_i\|_F^2$ is a constant w.r.t. $\b{\Theta}_i$ and $\b{Q}_i$ and can be dropped. 

\begin{figure*}[!t]
\centering
\includegraphics[width=6.5in]{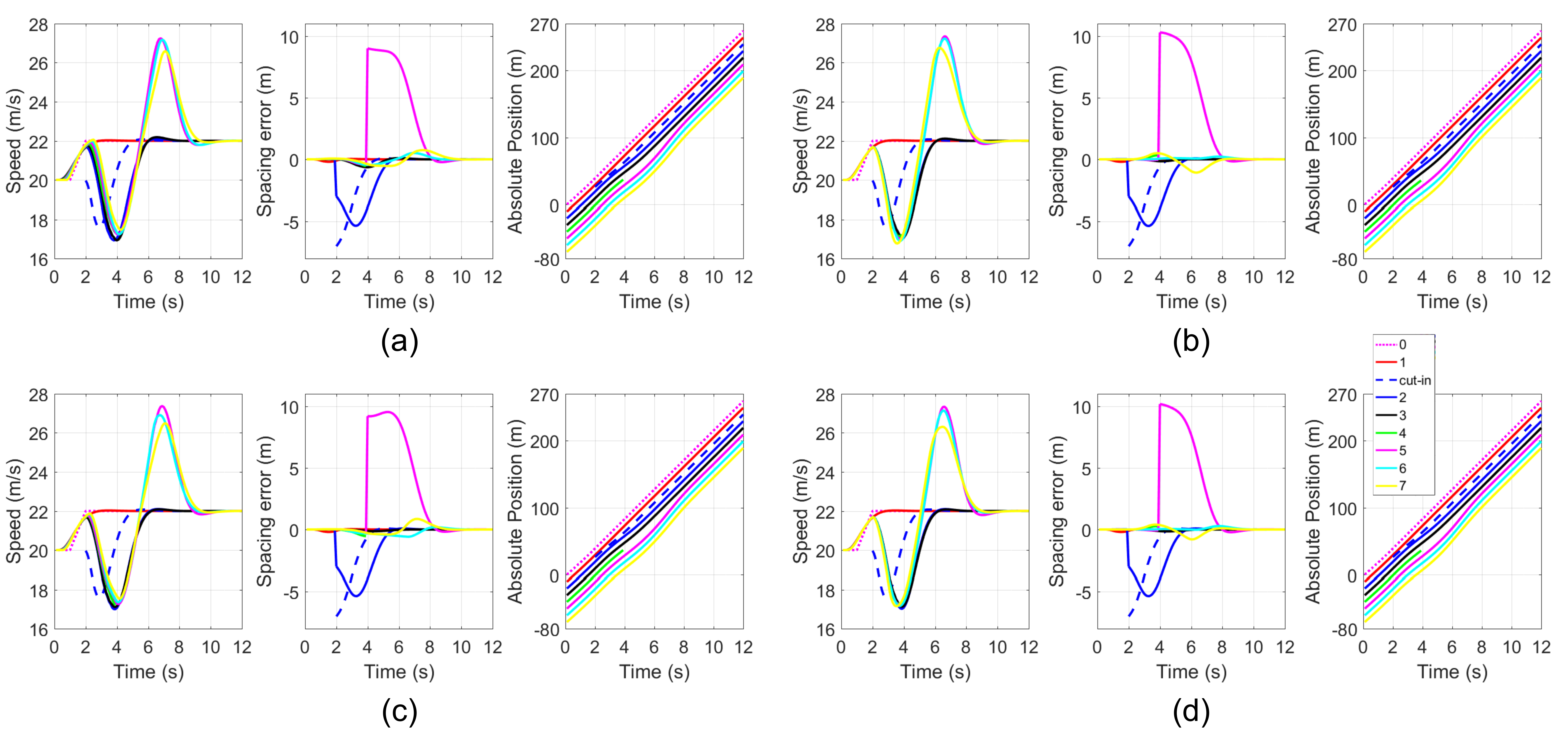}
\caption{Speed, relative spacing error with the preceding vehicle, and absolute position versus time: (a) PF, (b) PLF, (c) TPF, and (d) TPLF.}
\label{figure_results_convergence}
\end{figure*}

\begin{figure*}[!t]
\centering
\includegraphics[width=6.55in]{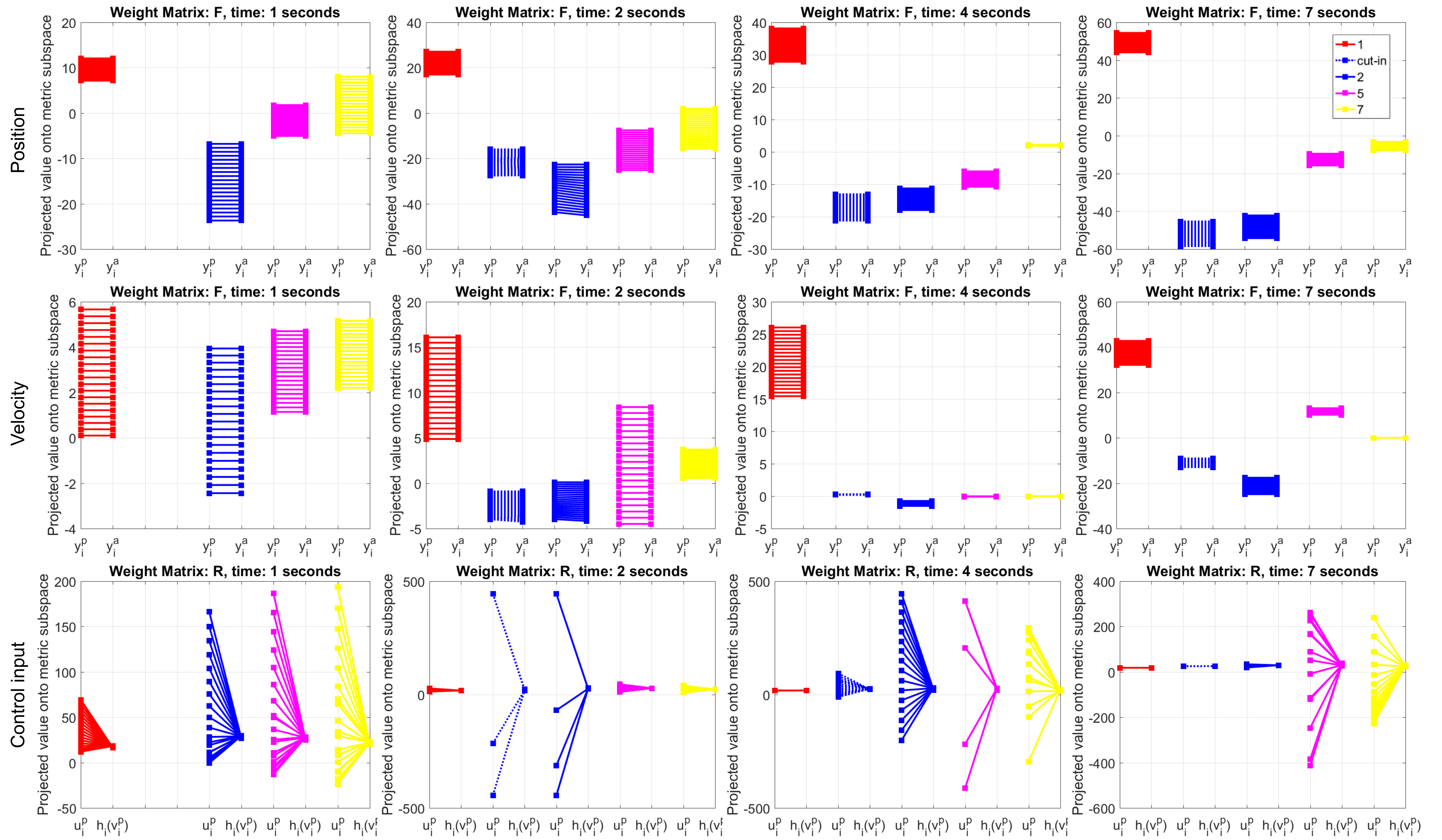}
\caption{Projected position, velocity, and control input values onto the metric subspaces (see Lemma \ref{lemma_metric_subspace_projection}) for driving comfort and fuel economy. The plots of weights $\b{F}$ and $R$ correspond to TPF and PF topologies, respectively.}
\label{figure_results_subspaces_behavior}
\end{figure*}

\begin{figure*}[!t]
\centering
\includegraphics[width=6.55in]{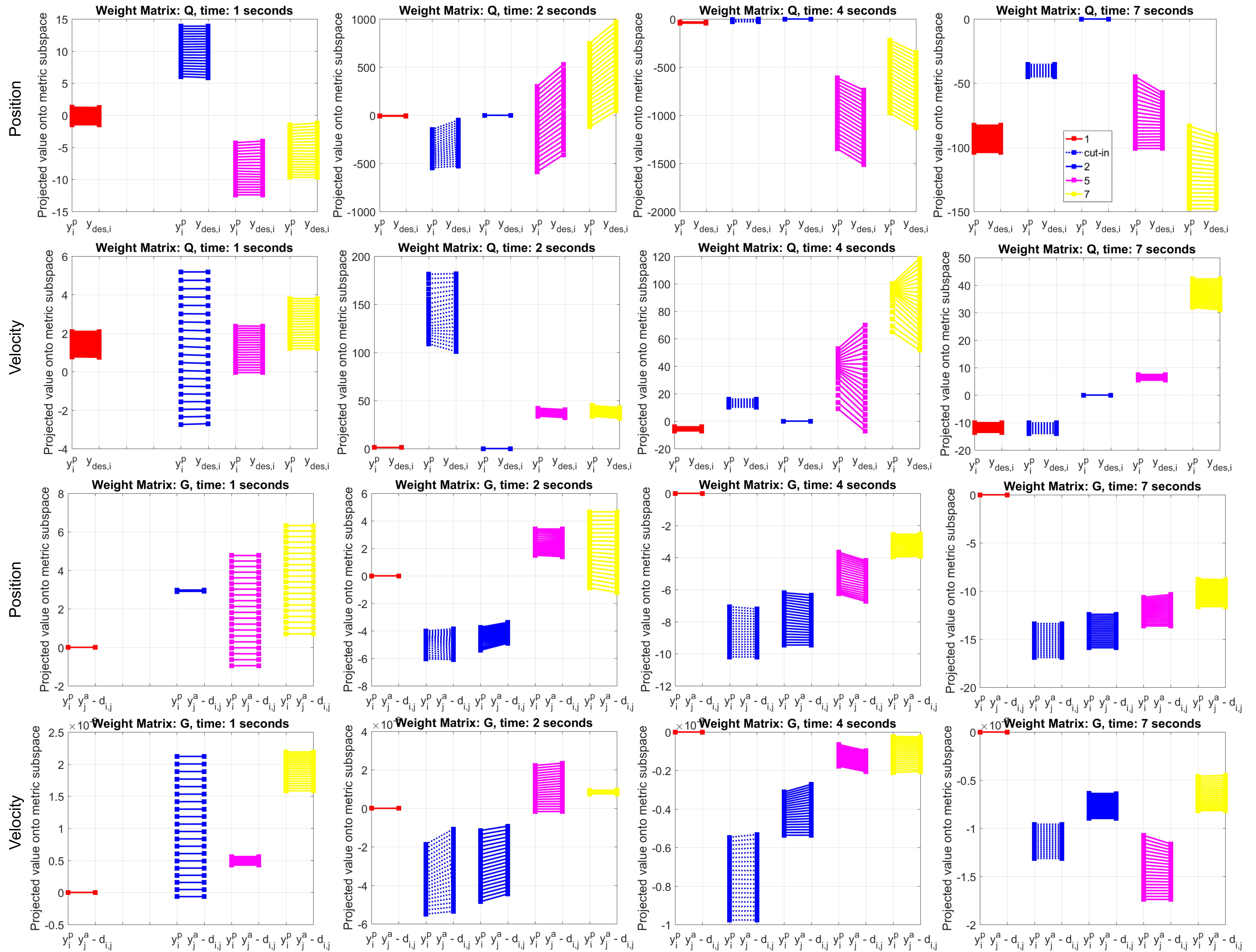}
\caption{Projected position and  velocity values onto the metric subspaces (see Lemma \ref{lemma_metric_subspace_projection}) for absolute and relative convergence. The plots of weights $\b{Q}$ and $\b{G}$ correspond to TPLF and PLF topologies, respectively.}
\label{figure_results_subspaces_convergence}
\end{figure*}

Consider the Lyapunov function for the $\kappa$-th iteration in ADMM {\cite[Appendix A]{boyd2011distributed}}: $V^{(\kappa)} := \, \frac{1}{\rho} \|\b{\Omega}_i^{(\kappa)} - \b{\Omega}_i^*\|_F^2 + \frac{1}{\rho} \|\b{\Delta}_{i,1}^{(\kappa)} - \b{\Delta}_{i,1}^*\|_F^2 + \frac{1}{\rho} \|\b{\Delta}_{i,2}^{(\kappa)} - \b{\Delta}_{i,2}^*\|_F^2 + \rho \|\b{\Theta}_i^{(\kappa)} - \b{\Theta}_i^*\|_F^2$ where the star superscript denotes the optimal primal and dual variables. 
Following the approach of {\cite[Lemma 4]{giesen2018distributed}} shows that this function is non-increasing. Therefore, the ADMM converges to the optimal solution for the primal and dual variables \cite{giesen2018distributed}, noticing that the problem has strong duality (see Proposition \ref{lemma_strong_duality}). 
\end{proof}

\begin{proposition}\label{lemma_ML_stability}
The platoon under Problem (\ref{equation_optimization_ML}) is asymptotically stable.
\end{proposition}
\begin{proof}
As the objective functions of Problems (\ref{equation_optimization_DNMPC}) and (\ref{equation_optimization_ML}) are equal, according to Lemma \ref{lemma_mainProblem_Lyapunov}, the objective in Problem (\ref{equation_optimization_ML}) is also a non-decreasing Lyapunov function.
Moreover, the constraints in Problem (\ref{equation_optimization_ML}) satisfy Corollaries \ref{corollary_metric_matrix_epsilon_positive_definite} and \ref{corollary_constraint_F}. Hence, the Problem (\ref{equation_optimization_ML}) is asymptotically stable.
\end{proof}

\begin{theorem}\label{theorem_ML_convergence_in_N_steps}
The solution of Problem (\ref{equation_optimization_ML}), obtained by Eq. (\ref{equation_ADMM_update}), converges to the solution of Problem (\ref{equation_optimization_DNMPC}) within at most $t_\text{conv}$ time steps (see Eq. (\ref{equation_DNMPC_convergence_in_N_steps_dynamic_time})), if satisfying Assumption \ref{assumption_spanning_tree}. 
\end{theorem}
\begin{proof}
According to Theorem \ref{theorem_ADMM_convergence}, Problem (\ref{equation_optimization_ML}) has a converged optimal solution. 
According to Theorem \ref{theorem_DNMPC_convergence_in_N_steps_dynamic}, the Problem (\ref{equation_optimization_DNMPC}) guarantees convergence without any restriction on the weight matrices except that they should belong to the positive semi-definite cone (see Corollary \ref{corollary_metric_matrix_epsilon_positive_definite}). 
According to Corollary \ref{corollary_espilon_positive_definite_subset}, the $\varepsilon$-positive definite cone is a subset of the positive semi-definite cone. 
Hence, the weight matrices have valid constraints for stability and convergence of the problem. 
Moreover, Eq. (\ref{equation_ADMM_update_U}) guarantees that the control input obtained from the ADMM updates satisfies Theorem \ref{theorem_DNMPC_convergence_in_N_steps_dynamic}. 
Noticing that the Problem (\ref{equation_optimization_DNMPC}) is stable, according to Proposition \ref{lemma_ML_stability}, the proof is complete. 
\end{proof}

\section{Simulations}\label{section_simulations}

\subsection{Synthetic Data}

For validating the proposed method, we made a heterogeneous platoon, including seven FVs ($N = 7$) at the initial time.
Following the dataset in \cite{zheng2017distributed}, we set the initial position and velocity of the leader as $s_0(0)=0$, $v_0(0)= 20$ m/s. The velocity of LV is considered as $20$ m/s, $20 + 2(t-1)$, and $22$ for $t \leq 1$ s, $t \in (1,2]$ s, and $t > 2$ s, respectively. 
The parameters of the FVs are set as in \cite{zheng2017distributed} (cf. {\cite[Table I]{zheng2017distributed}}). The sampling time is $\Delta t = 0.1$s, the horizon length is $N_p = 20$, and the desired gap is $d=10$m.

For the dynamic platooning, we introduced a cut-in between the first and second FVs at time $t=2$s and a cut-out of the fourth FV from the platoon at time $t=4$s. The mass, $\tau_i$, $C_{A,i}$, and $r_i$ of the cut-in vehicle were randomly set to be $1305.9$kg, $0.63$s, $1$, and $0.4$m, respectively. 

\subsection{Dynamic Platoon Control with Cut-in/Cut-out Maneuvers}

% Our implementation of the proposed method was built on top of the code, which can be found in \cite{YangDMPCcode}. 
The results of the simulations are shown in Fig. \ref{figure_results_convergence} where speed, relative spacing error with the preceding vehicle, and absolute position are illustrated for the four different topologies. 
As the absolute positions show, no collision has occurred in the platoon as expected. By reducing its speed, the second FV has increased its gap with the first FV to make the desired distance of $10$m from the cut-in vehicle. Consequently, the following vehicles have lessened their velocity to keep the desired distance. The plots of speed verify this fact. Moreover, the relative spacing error shows the jump in the distance error because of the cut-in maneuver. A similar analysis exists for the cut-out maneuver where the following vehicles have increased their velocity to reach the desired distance from the vehicles in front. 
As expected, the spacing error for the cut-out maneuver has an opposite sign with respect to the cut-in error. 
Furthermore, Fig. \ref{figure_results_convergence} shows that convergence has been reached in $t_\text{conv}=11$s which coincides with Theorem \ref{theorem_DNMPC_convergence_in_N_steps_dynamic} because $t_\text{conv}=\max(2,4)+7+1-1=11$s.

\subsection{Driving Behavior Analysis in the Dynamic Platoon}

Here, we analyze the driving behavior of the platoon. Four different behaviors, i.e., driving comfort, fuel economy, relative convergence, and absolute convergence, are analyzed. 
For the analysis, we use distributed metric learning, using ADMM optimization, to learn the metric subspaces (see Section \ref{section_metric_learning}). 
In the ADMM optimization, the weight matrices were initialized randomly in the $\varepsilon$-positive definite cone to be feasible. The number of iterations in both ADMM and gradient descent was $10$, and the learning rate and $\rho$ were both set to $0.1$. We used $\varepsilon = 0.01$ in our simulations.

\subsubsection{Driving Comfort} 

According to Eq. (\ref{equation_objective}), the metric with weight $\b{F}$ calculates the difference of $\b{y}_i^p$ and $\b{y}_i^a$. The less this difference is, the more comfortable the driving will be because the predicted and assumed position (and velocity) are closer.   
Figure \ref{figure_results_subspaces_behavior} (rows 1 and 2) depicts the projected values of predicted and assumed positions (and velocities) onto the metric subspace (see Lemma \ref{lemma_metric_subspace_projection}). 
In this figure, merely the values of the first, last, and most impacted FVs (cut-in car and cars 2 and 5) by the dynamic maneuvers are illustrated for the sake of brevity. 
The figures are provided for times $t=1$s (before cut-in), $t=2$s (at cut-in), $t=4$s (at cut-out), and $t=7$s (after cut-out). 
In the plots for the weight $\b{F}$, for every vehicle, the left and right points are the corresponding $\b{y}_i^p$ and $\b{y}_i^a$ in the horizon connected by lines to show their difference.
Similar notation is used for the plots of the metrics with other weight matrices explained in the following subsections.
For every weight matrix, the subspace of one of the topologies is shown due to the lack of space.

For better driving comfort, the lines should be more horizontal to have less difference between the predicted and assumed outputs. Also, more vertically compact points show a smoother change in the horizon. Hence, more horizontal lines and vertically compact plots indicate more comfort (similar analysis exists for the subspaces of other weights). 
As seen in Fig. \ref{figure_results_subspaces_behavior} (rows 1 and 2), chaos in comfort has occurred for FV 2 and 7, caused by the cut-in and cut-out maneuvers, at times $t=2$s and $t=4$s, respectively. However, the driving comfort is improved by passing the time and progress in the algorithm. 

\subsubsection{Fuel Economy} 

As in Eq. (\ref{equation_objective}), the metric having $R$ as its weight measures the difference of the predicted control input ($u_i^p$) from its equilibrium control input ($h_i(v_i^p)$). The larger difference requires more fuel consumption because of more abrupt changes in the control input; hence, projection onto the subspace of this metric indicates the fuel economy. In Fig. \ref{figure_results_subspaces_behavior} (row 3), the changes of the equilibrium control input are small, as expected, but the chaos caused by the dynamic maneuvers results in more sparse changes in predicted control input for the most affected vehicles. 

\subsubsection{Absolute and Relative Convergence} 

The metrics with weights $\b{Q}$ and $\b{G}$, in Eq. (\ref{equation_objective}), are responsible for the absolute and relative convergence because they measure the difference of the predicted output ($\b{y}_i^p$) from the desired output ($\b{y}_{\text{des},i}$) and shifted output of the neighbor vehicles ($\b{y}_j^a - \widetilde{\b{d}}_{i,j}$), respectively. 
In Fig. \ref{figure_results_subspaces_convergence} for both of subspaces, we see the lines of both position and velocity plots become less horizontal and compact at the times of dynamic maneuvers. However, the progress of the algorithm alleviates the effect of chaos. 

\section{Conclusion and Future Directions}\label{section_conclusion}

Dynamic heterogeneous vehicle platooning is one of the important tasks in autonomous control. In this paper, we proposed a DNMPC-based approach, based on the technique introduced in \cite{zheng2017distributed}, for handling possible cut-in/cut-out maneuvers and avoiding collisions by tracking the desired velocity and maintaining the safe desired gap among the vehicles.
We derived the convergence time of the DNMPC for dynamic platooning based on the time of maneuvers. Furthermore, we analyzed driving experience factors such as driving comfort, fuel economy, and absolute and relative convergence of the method using distributed metric learning and ADMM optimization. Our simulations on a dynamic platoon with cut-in and cut-out maneuvers with different topologies validated the effectiveness of the method. As a future direction, the string stability can also be analyzed in the DNMPC. Furthermore, the method can be generalized by considering lateral vehicle dynamics.

\section*{Acknowledgment}
The authors would like to thank Dr. Yang Zheng (SEAS and CGBC at Harvard University, Cambridge, MA, USA) for the fruitful discussions during this work. 
This work is partially supported by NSERC of Canada.
% Also, the authors would like to thank the Natural Sciences and Engineering Research Council (NSERC) of Canada for their support of this research.

\bibliographystyle{IEEEtran}
\balance
\bibliography{references}

% \begin{thebibliography}{99}

% \bibitem{c1}
% J.G.F. Francis, The QR Transformation I, {\it Comput. J.}, vol. 4, 1961, pp 265-271.

% \bibitem{c2}
% H. Kwakernaak and R. Sivan, {\it Modern Signals and Systems}, Prentice Hall, Englewood Cliffs, NJ; 1991.

% \bibitem{c3}
% D. Boley and R. Maier, "A Parallel QR Algorithm for the Non-Symmetric Eigenvalue Algorithm", {\it in Third SIAM Conference on Applied Linear Algebra}, Madison, WI, 1988, pp. A20.

% \end{thebibliography}

\end{document}